\documentclass[11pt,a4paper]{article}
\usepackage{color}
\usepackage{graphicx}
\usepackage{amsfonts,amsmath,amssymb}
\newcommand{\e}{\rm{e}}

\newcommand{\bE}{{\bf E}}
\newcommand{\bR}{{\bf R}}

\newcommand{\be}{\begin{equation}}
\newcommand{\ee}[1]{\end{equation}{\label{#1}}}
\newcommand{\rf}[1]{\eqref{#1}}
\newcommand{\wh}[1]{{\widehat{#1}}}

\newcommand{\bqn}{\begin{eqnarray}}
\newcommand{\eqn}[1]{\label{#1}\end{eqnarray}}

\newcommand{\bse}{\begin{eqnarray*}}
	\newcommand{\ese}{\end{eqnarray*}}

\newcommand{\nn}{\nonumber \\}

\newcommand{\la}{\langle}
\newcommand{\ra}{\rangle}
\newcommand{\Prox}{{\rm Prox}}

\newcommand{\Prob}{{\rm Prob}}

\newcommand{\cA}{{\cal A}}
\newcommand{\cB}{{\cal B}}

\newcommand{\ovN}{ {\overline{N}}}

\newcommand{\half}{ \mbox{\small$\frac{1}{2}$}}
\newcommand{\argmin}{{\rm arg\,min}}

\def\qed{\ \hfill$\square$\par\smallskip}
\newtheorem{proposition}{Proposition}%[section]
\newtheorem{corollary}{Corollary}%[section]
\newtheorem{theorem}{Theorem}%[section]

\newtheorem{thm}{Theorem}%[section]

\newtheorem{lemma}[thm]{Lemma}

\newenvironment{proofoflemma}[1]
{\vspace{1pt}\par{\bf%
		Proof of Lemma~#1.}}%
{\noindent\vspace{1pt}}

\newenvironment{proof}
{\vspace{1pt}\par{\bf%
		Proof\,}}%
{\noindent\vspace{17pt}}

\begin{document}  %%%!!!
	
	%	\year{2019}
	\title{%АЛГОРИТМЫ РОБАСТНОЙ СТОХАСТИЧЕСКОЙ ОПТИМИЗАЦИИ
		%	НА ОСНОВЕ МЕТОДА\\ ЗЕРКАЛЬНОГО СПУСКА%
		Algorithms of Robust Stochastic Optimization
		Based on Mirror Descent Method%
	}%
	%\thanks{The work of A.V.~Nazin was supported by the Russian Science Foundation (Grant No. 16-11-10015).
	%	The work of A.B.~Tsybakov was supported by the GENES Institute and
	%by grant
	%Labex Ecodec (ANR-11-LABEX-0047).
	%	The work of A.B.~Juditsky was supported by grant PGMO 2016-2032H and jointly by A.B.~Juditsky with A.S.~Nemirovsky ---	NSF grant CCF-1523768.
	%	%Работа А.Б.~Юдицкого поддержана грантом PGMO 2016-2032H и совместно А.Б.~Юдицкого с А.С.~Немировским  ---
	%	%грантом NSF  CCF-1523768.
	%	%Работа А.В.~Назина поддержана Российским научным фондом (грант №16-11-10015).
	%	%Работа А.Б.~Цыбакова поддержана институтом GENES и
	%	%грантом
	%	%IPANEMA (ANR-13-BSH1-0004-02) и
	%}
	\date{}
	
	\author{
A.B.~Juditsky,\thanks{LJK,
		Universit\'e Grenoble Alpes, Grenoble
		{\tt anatoli.juditsky@univ-grenoble-alpes.fr}}
		\and
A.V.~Nazin,\thanks{Institue of Control Sciences RAS, Moscow {\tt nazine@ipu.ru}}\and
A.S.~Nemirovsky,\thanks{ISyE, Georgia Institute of Technology, Atlanta {\tt nemirovs@isye.gatech.edu}}\and
A.B.~Tsybakov\thanks{CREST, ENSAE {\tt alexandre.tsybakov@ensae.fr}\newline
Work of A.V.~Nazin was supported by the Russian Science Foundation (Grant No. 16-11-10015), work of A.B.~Tsybakov was supported by the GENES Institute and by the grant
		Labex Ecodec (ANR-11-LABEX-0047).
		A.B.~Juditsky was supported by the grant PGMO 2016-2032H and A.B.~Juditsky and A.S.~Nemirovsky were supported by the NSF grant CCF-1523768.
	}}
	
	\maketitle
	
	\begin{abstract}	
		We propose an approach to the construction of robust non-Euclidean iterative algorithms by convex composite stochastic optimization based on truncation of stochastic gradients.
		For such algorithms, we establish sub-Gaussian confidence bounds under weak assumptions about the tails of the noise distribution in convex and strongly convex settings. Robust estimates of the accuracy of general stochastic algorithms are also proposed.
		%We discuss an approach to constructing robust iterative algorithms of convex composite stochastic
		%optimization based on truncation of stochastic gradients. We show that such algorithms satisfy
		%``standard'' sub-Gaussian confidence guarantees under weak assumptions on the tails of the noise
		%distribution in the convex and strongly convex settings.
		%We also propose %a robust
		%a construction
		%of robust estimates of the accuracy of approximate solutions of stochastic algorithms.
	\end{abstract}

	\textit{Keywords}:
	%\KeyWords:
	robust iterative algorithms,
	stochastic optimization algorithms,
	convex composite stochastic optimization,
	mirror descent method,
	robust confidence sets.

	\section{Introduction}
	
	In this paper, we consider the problem of {\em convex composite stochastic  optimization}:
	\bqn
	\min\limits_{x\in X} F(x), \qquad \qquad F(x) =\bE\{\Phi(x,\omega)\}+\psi(x),
	\eqn{prob1}
	where $X$ is a compact convex subset of
	a finite-dimensional real vector space
	$E$ with norm $\|\cdot\|$,  $\omega$ is a random variable on a probability space $\Omega$ with distribution $P$, function $\psi$ is convex and continuous,
	and function $\Phi:\;X\times \Omega\to \bR$.
	Suppose that the expectation
	\[
	\phi(x):=\bE\{\Phi(x,\omega)\}=\int_\Omega \Phi(x,\omega) dP(\omega)
	\]
	is finite for all $x\in X$,
	and is a convex and differentiable function of $x$.
	Under these assumptions, the problem \rf{prob1} has a solution
	with optimal value $F_*=\min_{x\in{X}}F(x)$.
	\par
	Assume that there is an oracle, which
	for any input $(x,\omega) \in X \times \Omega$
	returns a stochastic gradient that is  a vector $G(x,\omega)$
	satisfying
	\bqn
	\bE\{G(x,\omega)\}=\nabla \phi(x) \quad\mbox{and}\quad \bE\{\|G(x,\omega)-\nabla \phi(x)\|^2_*\}\leq \sigma^2,
	\quad \forall\,x\in X,
	\eqn{eq:soracle}
	where $\|\cdot\|_*$ is conjugate norm to $\|\cdot\|$, and $\sigma>0$ is
	a constant.
	The aim of this paper is to construct $ (1- \alpha)$-reliable approximate solutions of the problem \rf{prob1}, i.e., solutions
	$\wh{x}_N$, based on $ N $ queries of the oracle and satisfying the condition
	\bqn
	\Prob\{F(\wh x_N)-F^*\leq \delta_N(\alpha)\}\geq 1-\alpha,
	\quad\forall\,\alpha\in (0,1),
	\eqn{eq:conf1}
	with  as small as possible $\delta_N(\alpha)>0$.
	\par
	Note that stochastic optimization problems of the form \rf{prob1} arise in the context of penalized risk minimization, where the confidence bounds \rf{eq:conf1} are directly converted into confidence bounds for the accuracy of the obtained estimators. In this paper, the bounds \rf{eq:conf1} are derived with $\delta_N(\alpha)$ of order $\sqrt{\ln(1/\alpha)/N}$.
	Such bounds are often called sub-Gaussian confidence bounds. Standard results on sub-Gaussian confidence bounds for stochastic optimization algorithms assume boundedness of exponential or subexponential moments of the stochastic  noise of the oracle  $G(x,\omega)-\nabla \phi(x)$
	(cf.  \cite{nemirovski2009robust,juditsky2014deterministic,ghadimi2012optimal}).
	In the present paper, we propose robust stochastic algorithms that satisfy sub-Gaussian bounds of type \rf{eq:conf1} under a significantly less restrictive condition \rf{eq:soracle}.
	\par
	Recall that the notion of robustness of statistical decision procedures was introduced by J.~Tukey \cite{tukey1960survey} and P.~Huber %[5--7]
	\cite{huber1964robust,huber1972,huber1981robust}
	in the $1960$ies, which led to the subsequent development of robust stochastic approximation algorithms.
	In particular, in the 1970ies--1980ies, algorithms that are robust for wide classes of noise distributions were proposed for problems of stochastic optimization and parametric identification.
	Their asymptotic  properties when the sample size  increases have been well studied, see, for example, %[8--16]
	\cite{martin1975robust,polyak1979adaptive,polyak1980robust,poljak1980robusta,price1979robust,stankovic1986analysis,chen1987convergence,chen1989robustness,nazin1992optimal} and references therein.
	An important contribution to the development of the robust approach was made by Ya.Z.~Tsypkin.
	Thus, a significant place in the monographs \cite{YaZTsypkin1984,YaZTsypkin1995} is devoted to the study of iterative robust identification algorithms.
	\par
	The interest in robust estimation resumed in the 2010ies due to the need to develop statistical procedures that are resistant to noise with heavy tails in high-dimensional problems.
	{Some recent work
		%\cite{bubeck2013bandits,devroye2016sub,lugosi2019sub}
		%\marginpar{Lecue-etal}
		%[19--21],
		%и \marginpar{Check!}
		%[19--29]
		\cite{Lecue-etal-19,Lecue-etal-20,Lecue-etal-21,Lecue-etal-22,Lerale-etal-23,Lugosi2016Mendelson,Lugosi2017Mendelson,Lugosi2018Mendelson%}
			,
			hsu2016loss,bubeck2013bandits,devroye2016sub}
		develops the method of median of means \cite{nemirovskii1979complexity} for constructing estimates that satisfy sub-Gaussian confidence bounds for noise with heavy tails.
		Thus, in \cite{hsu2016loss} the median of means approach was used to construct an $(1-\alpha)$-reliable version of stochastic approximation with averaging (``batch'' algorithm) in a stochastic optimization setting similar to \rf{prob1}.
		Other original approaches were developed in
		%[31--35]%
		%\cite{bubeck2013bandits,devroye2016sub,lugosi2019sub},
		%%\cite{Lecue-etal-19,hsu2016loss},
		\cite{lugosi2019sub,catoni2012challenging,audibert2011robust,minsker2015geometric,wei2017estimation}%
		, in particular, the geometric median techniques for robust estimation of signals and covariance matrices with sub-Gaussian guarantees \cite{minsker2015geometric,wei2017estimation}.
		Also there was a renewal of interest in robust iterative algorithms. Thus, it was shown that robustness of stochastic approximation algorithms can be enhanced by using the geometric median of stochastic gradients  \cite{chen2017distributed,yin2018byzantine}.
		Another variant of the stochastic approximation procedure for calculating the geometric median was studied in \cite{cardot2010stochastic,cardot2017online}, where a specific property of the problem (boundedness of the stochastic gradients) allowed the authors to construct $(1-\alpha)$-reliable bounds under a very weak assumption about the tails of the noise distribution.
	}
	\par
	This paper discusses an approach to the construction of robust stochastic algorithms based on truncation of the stochastic gradients. It is shown that this method satisfies sub-Gaussian confidence bounds.
	In Sections 2 and 3, we define the main components of the optimization problem under consideration.
	In Section 4, we define the robust stochastic mirror descent algorithm and establish confidence bounds for it. Section 5 is devoted to robust accuracy estimates for general stochastic algorithms.
	Finally, Section 6 establishes robust confidence bounds for problems, in which $F$ has a quadratic growth. The Appendix contains the proofs of the results of the paper.
	
	\section{Notation and Definitions}
	
	\label{sec:pstat}
	
	Let $E$ be a
	finite-dimensional real vector space with norm $\|\cdot\|$ and let $E^*$ be the conjugate space to $E$.
	Denote by $\langle s, x \rangle$ the value of linear function
	$s \in E^*$ at point $x \in E$ and by
	$\| \cdot \|_*$ the conjugate to norm $\|\cdot\|$ on $E^*$, i.e.,
	\[
	\| s \|_* = \max\limits_x \{ \la s, x \ra: \; \|x\|\leq 1 \},
	\quad s \in E^*.
	\]
	On the unit ball
	\[B=\{x\in E:\;\|x\|\leq 1\},\]
	we consider a continuous convex function $\theta:B\to\textbf{R}$ with the following property:
	\bqn
	\langle\theta'(x)-\theta'(x'),x-x'\rangle\ge \|x-x'\|^2,\quad \forall x,x'\in B,
	\eqn{pp}
	where $\theta'(\cdot)$ is a continuous in
	$B^o = \{x\in B: \, \partial\theta(x)\neq\emptyset\}$ version of the
	subgradient of $\theta(\cdot)$ and $\partial\theta(x)$ denotes the subdifferential of function $\theta(\cdot)$ at point $x$, i.e., the set of all subgradients at this point.
	In other words, function $\theta(\cdot)$ is strongly convex on $B$ with  coefficient~1 with respect to the norm $\|\cdot\|$.
	We will call $\theta(\cdot)$ the normalized proxy function.
	Examples of such functions are:
	\begin{itemize}
		\item  $\theta(x)=\half \|x\|_2^2$ for $\left(E, \|\cdot\|\right)=\left(\textbf{R}^n, \|\cdot\|_2\right)$;
		\item  $\theta(x)={2\e(\ln n)} \|x\|_p^p$ with $p=p(n):=1+{1\over 2\ln n}$ for
		$\left(E, \|\cdot\|\right)=\left(\textbf{R}^n, \|\cdot\|_1\right)$;
		\item $\theta(x)=4\e(\ln n)\sum_{i=1}^n |\lambda_i(x)|^p$\, with\, $p=p(n)$
		for $E=S_n$, where $S_n$ is the space of symmetric $n\times n$ matrices equipped with the nuclear norm $\|x\|=\sum_{i=1}^n |\lambda_i(x)|$ and $\lambda_i(x)$ are eigenvalues of matrix $x$.
	\end{itemize}
	%Hereinafter
	Here and in what follows, $\|\cdot\|_p$ denotes the $\ell_p$-norm in $\mathbf{R}^n$, $p\geq1$.
	Without loss of generality, we will assume below that
	\[
	0=\argmin_{x\in B}\,\theta(x).
	\]
	We also introduce the notation
	\[
	\Theta=\max_{x\in B}\theta(x)-\min_{x\in B}\theta(x)\geq \half.
	\]
	Now, let $X$ be a convex compact subset in $E$ and let $x_0\in X$ and $R>0$ be such that
	$\max_{x\in X}\|x-x_0\|\leq R$. We equip $X$ with a proxy function
	\[
	\vartheta(x)=
	%\vartheta^{x_0,R}(x)=
	R^2\theta\left({x-x_0\over R}\right).
	\]
	Note that $\vartheta(\cdot)$ is strongly convex with coefficient~1 and
	\[
	\max_{x\in X}\vartheta(x)-\min_{x\in X}\vartheta(x)\leq R^2\Theta.
	\]
	Let $D:=\max_{x,x'\in X}\|x-x'\|$ be the diameter of the set $X$. Then $D\leq 2R$.
	\par
	We will also use the Bregman divergence
	\[
	V_x(z)=\vartheta(z)-\vartheta(x)-\langle \vartheta'(x),z-x\rangle,
	\quad\forall\, z, x\in{X}.
	\]
	In the following, we denote by $C$ and $C'$ positive numerical constants, not necessarily the same in different cases.
	%%%%%%%%%%
	
	\section{Assumptions}\label{sec:pstat}
	
	Consider a convex  composite stochastic optimisation problem \rf{prob1} on
	a convex compact set $X\subset E$.
	Assume in the following that the function
	\[
	\phi(x)=\bE\{\Phi(x,\omega)\}
	\]
	is convex on $X$, differentiable at each point of the set $X$ and its
	gradient satisfies the Lipschitz condition
	\bqn
	\|\nabla \phi(x')-\nabla \phi(x)\|_*\leq L\|x-x'\|,\qquad \forall \,x,x'\in X.
	\eqn{Lipf}
	\par
	Assume also that function $\psi$ is convex and continuous.
	In what follows, we assume that we have at our disposal a stochastic oracle, which for
	any input
	$(x,\omega)\in X\times \Omega$,
	returns a random vector
	$G(x,\omega)$,
	satisfying the conditions \eqref{eq:soracle}.
	In addition, it is assumed that for any $a\in E^*$ and $\beta>0$ an exact solution of the minimization problem
	\[
	\min_{z\in X} \{\langle a,z\rangle +\psi(z) +\beta\vartheta(z) \}
	\]
	is available.
	This assumption is fulfilled for typical  penalty functions $\psi$, such as convex power functions of  the $\ell_p$-norm
	(if $X$ is a convex compact in $\mathbf{R}^n$)
	or negative entropy $\psi(x) = \kappa \sum_{j=1}^{n}  x_j \ln x_j$, where $\kappa>0$
	(if $X$ is the standard simplex in $\mathbf{R}^n$).
	Finally, it is assumed that a vector $g(\bar{x})$ is available,
	where $\bar{x}\in X$ is a point in the set $X$ such that
	\begin{equation}\label{eq:8a}
	\|g(\bar{x}) - \nabla\phi(\bar{x})\|_* \leq \upsilon\sigma
	\end{equation}
	with a constant $\upsilon\geq0$.
	This assumption is motivated as follows.
	\par
	First, if we a {\em priori} know that the global minimum of function $\phi$
	is attained at an interior point $ x_\phi$ of the set $X$ (what is common in statistical applications of stochastic approximation),
	we have $\nabla\phi(x_\phi)=0$. Therefore, choosing $\bar x=x_\phi$,
	one can put $g(\bar x)=0$ and assumption \rf{eq:8a} holds automatically with $\upsilon=0$.
	\par
	Second, in general, one can choose $\bar x$
	as any point of the set $X$ and $g(\bar x)$
	as a geometric median of stochastic gradients
	$G(\bar x,\omega_i)$, $i=1,\dots,m$, over $m$ oracle queries.
	It follows from
	\cite{minsker2015geometric}
	that if $m$ is of order $\ln\left(\varepsilon^{-1}\right)$
	with some sufficiently small $\varepsilon>0$, then
	\bqn
	\Prob\{\|g(\bar x)-\nabla \phi(\bar x)\|_*> \upsilon\sigma \}\leq \varepsilon.
	\eqn{bar g}
	Thus, the confidence bounds obtained below will remain valid up to an $\varepsilon$-correction in the probability of deviations.

	%%%%%%%%%%%%%%%%%%%%%%%%%%%%%%%%
	
	\section{Accuracy bounds for Algorithm RSMD%РСЗС
	}
	
	\label{sec:smd}
	
	In what follows, we consider
	that the assumptions of Section \ref{sec:pstat} are fulfilled.
	Introduce a composite proximal transform
	\begin{eqnarray}\label{prox}
	\Prox_{\beta,x}(\xi)&:=&\argmin_{z\in X}\big\{\langle  \xi,z\rangle +\psi(z)+\beta V_{x}(z)\big\}=\\
	\nonumber &=&\argmin_{z\in X}\big\{\langle  \xi-\beta \vartheta'(x),z\rangle +\psi(z)+\beta \vartheta(z)\big\},
	\end{eqnarray}
	where $\beta>0$ is a tuning parameter.
	\par
	For $i=1,2,\dots$,
	define the algorithm of
	{\em Robust Stochastic Mirror Descent} (RSMD) by the recursion
	\bqn
	x_{i}=\Prox_{\beta_{i-1},x_{i-1}}(y_{i}),\;\;\;x_0\in X,
	\eqn{eq:md1}
	\bqn
	~~~~~~y_i=\left\{\begin{array}{ll}G(x_{i-1},\omega_i),&\mbox{if}\;\; \|G(x_{i-1},\omega_i)-g(\bar x)\|_*\leq L\|\bar x-x_{i-1}\|
		+\lambda+\upsilon\sigma,\\
		g(\bar x),&\mbox{otherwise.}
	\end{array}
	\right.
	\eqn{t_grad0}
	Here $\beta_i>0$, $i=0,1,\dots$, and $\lambda>0$ are tuning parameters that will be defined below,
	and $\omega_1,\omega_2,\dots$ are independent identically distributed (i.i.d.) realizations of a random variable $\omega$,
	corresponding to the oracle queries at each step of the algorithm.
	
	The approximate solution of problem \rf{prob1} after $N$ iterations is defined as the weighted average
	\bqn
	\widehat{x}_N=\left[\sum_{i=1}^{N}\beta^{-1}_{i-1}\right]^{-1}\sum_{i=1}^{N} \beta_{i-1}^{-1}x_i.
	\eqn{eq:asol}
	
	\par
	If the global minimum of function $\phi$ is attained at an interior point of the set $X$ and $\upsilon=0$,
	then definition \rf{t_grad0} is simplified.
	In this case, replacing $\|\bar{x}-x_{i-1}\|$ by the upper bound $D$ and putting $\upsilon=0$ and $g(\bar x)=0$ in \rf{t_grad0}, we define the truncated stochastic gradient by the formula
	\[
	y_i=\left\{\begin{array}{ll}G(x_{i-1},\omega_i),&\mbox{if~} \,\,\|G(x_{i-1},\omega_i)\|_*\leq LD+\lambda,\\
	0,&\mbox{otherwise}.\end{array}
	\right.
	\]
	\par
	The next result describes some useful properties of mirror descent recursion \rf{eq:md1}.
	Define
	\[ \xi_i=y_i-\nabla \phi(x_{i-1})\]
	and
	\bqn
	\varepsilon(x^N,z)=\sum_{i=1}^{N}\beta^{-1}_{i-1}[\langle \nabla \phi(x_{i-1}),x_{i}-z\rangle+ \psi(x_{i})-\psi(z)]+\half V_{x_{i-1}}(x_i),
	\eqn{epsdef}
	where $x^N=(x_0,\dots, x_N)$.
	
	\begin{proposition}\label{pr:muprop}
		Let $\beta_i\ge 2L$ for all $i=0,1,...$, and let  $\widehat{x}_N$ be defined in \rf{eq:asol}, where $x_i$
		are iterations \rf{eq:md1} for any values $y_i$, not necessarily given by  \rf{t_grad0}. Then for any $z\in X$ we have
		\bqn
		\left[\sum_{i=1}^{N}\beta^{-1}_{i-1}\right][F(\widehat{x}_N)-F(z)]&\le&\sum_{i=1}^{N}\beta^{-1}_{i-1}[ F(x_i)-F(z)]\leq
		\varepsilon(x^N,z)
		\nn
		&\leq & V_{x_0}(z)+\sum_{i=1}^{N}\Big[
		{\langle \xi_{i},z-x_{i-1}\rangle\over \beta_{i-1}} +{\|\xi_{i}\|_*^2\over \beta_{i-1}^2}\Big]
		\label{simple0}\\
		&\leq & 2V_{x_0}(z)+\sum_{i=1}^{N}\Big[
		{\langle \xi_{i},z_{i-1}-x_{i-1}\rangle\over \beta_{i-1}} +{ 3\over 2}{\|\xi_{i}\|_*^2\over \beta_{i-1}^2}\Big],
		\eqn{simple1}
		where $z_i$ is a random vector with values in $X$ depending only on $x_0,\xi_1,\dots,\xi_i$.
	\end{proposition}
	Using Proposition \ref{pr:muprop} we obtain the following bounds on the expected error $F(\wh{x}_N)-F_*$ of the approximate solution of problem \rf{prob1} based on the RSMD  algorithm. In what follows, we denote by $\bE\{\cdot\}$ the expectation with respect to the distribution of $\omega^N=(\omega_1,...,\omega_N)\in \Omega^{\otimes N}$.
	\begin{corollary}\label{cor:secmom}
		Set $M=LR$. Assume that
		$
		\lambda\geq \max\{M,\sigma\sqrt{N}\}+\upsilon\sigma$ and $\beta_i\geq 2L$ for all $i=0,1,...$.
		Let  $\wh{x}_N$ be the approximate solution \rf{eq:asol}, where $x_i$ are the iterations of the RSMD algorithm
		defined by relations \rf{eq:md1} and \rf{t_grad0}.
		Then
		\bqn
		\bE\{F(\wh{x}_N)\}- F_*&\leq&
		\left[\sum_{i=1}^{N}\beta^{-1}_{i-1}\right]^{-1}\left[R^2\Theta+
		\sum_{i=1}^N\left({2D\sigma\over \beta_{i-1}\sqrt{N}}+{4\sigma^2\over\beta_{i-1}^2}
		\right)\right].
		\eqn{eq:cor1-1}
		In particular, if
		$\beta_i=\bar\beta$ for all $i=0,1,...$, where
		\bqn
		\bar\beta=\max\left\{2L,{\sigma\sqrt{N}\over R\sqrt{\Theta}}\right\},
		\eqn{eq:betabar}
		then the following inequalities hold:
		\bqn
		\bE\{F(\wh{x}_N)\}- F_*\leq \mbox{$\bar\beta\over N$}\bE\left\{\sup_{z\in X}\varepsilon(x^N,z)\right\}
		\leq
		%\left[\sum_{i=1}^{N}\beta^{-1}_{i-1}\right]^{-1}
		%\left[2R^2\Theta+
		%\sum_{i=1}^N\left({2D\sigma\over \beta_{i-1}\sqrt{N}}+{6\sigma^2\over\beta_{i-1}^2}\right)\right]\leq
		C\max\left\{{LR^2\Theta\over N},\,{\sigma R\sqrt{\Theta}\over \sqrt{N}}\right\}.
		\eqn{eq:cor1supz}
		Moreover, in this case we have the following inequality with explicit constants:
		\[
		\bE\{F(\wh{x}_N)\}- F_*\leq\max\left\{{2LR^2\Theta\over N}+{4R\sigma(1+\sqrt{\Theta})\over \sqrt{N}},
		\;{2R\sigma(1+4\sqrt{\Theta})\over \sqrt{N}}\right\}.
		\]
	\end{corollary}
	This result shows that if the truncation threshold $\lambda$ is large enough, then the expected error of the proposed algorithm is bounded similarly to the expected error of the standard mirror descent algorithm with averaging, i.e., the algorithm in which stochastic gradients are taken without truncation: $y_i=G(x_{i-1},\omega_i)$.
	
	The following theorem gives confidence bounds for the proposed algorithm.
	\begin{theorem}\label{th:coset1}
		Let
		$\beta_i=\bar\beta\geq 2L$ for all $i=0,1,...$, and let $1\leq \tau\leq N/\upsilon^2$,
		\bqn
		\lambda=\max\left\{\sigma\sqrt{N\over \tau}, M\right\}+\upsilon \sigma.
		\eqn{eq:lambda}
		Let  $\wh{x}_N$ be the approximate solution \rf{eq:asol}, where $x_i$ are the RSMD iterations  defined by relations \rf{eq:md1} and \rf{t_grad0}.
		Then there is a random event $\cA_N\subset \Omega^{\otimes N}$ of probability at least $1-2e^{-\tau}$ such that for all $\omega^N\in \cA_N$ the following inequalities hold:
		\bse
		F(\wh{x}_N)- F_*&\leq &
		\frac{\bar\beta}{N} \sup_{z\in X}\varepsilon(x^N,z)\leq\\
		&\leq&\frac{C}{N}\left({\bar\beta R^2\Theta}+ R\max\left\{\sigma\sqrt{\tau  N},{M\tau}\right\}+{\bar\beta^{-1}\max\{N\sigma^2,M^2\tau\}}\right).
		\ese
		In paticular, chosing $\bar\beta$ as in formula \rf{eq:betabar} we have,
		for all $\omega^N\in \cA_N$,
		\bqn
		~~~~~F(\wh{x}_N)- F_*\leq
		\max\left\{C_1{LR^2[\tau\vee\Theta]\over N},\;C_2\sigma R\sqrt{\tau\vee\Theta \over N}\right\},
		\eqn{eq:thm12}
		where $C_1>0$ and $C_2>0$ are numerical constants.
	\end{theorem}
	
	The values of the numerical constants $C_1$ and $C_2$ in \rf{eq:thm12} can be obtained from the proof of the theorem, cf. the bound in \rf{c'c''}.
	
	Confidence bound \rf{eq:thm12} in Theorem \ref{th:coset1}
	contains two terms corresponding to the deterministic error and to the stochastic error. Unlike the case of noise with a ``light tail'' (see, for example, \cite{lan2012optimal}) and the bound in expectation \rf{eq:cor1supz}, the deterministic error  $LR^2[\tau\vee\Theta]/N$ depends on $\tau$.
	Note also that Theorem \ref{th:coset1} gives a sub-Gaussian confidence bound (the order of the stochastic error is $\sigma R\sqrt{[\tau\vee\Theta]/N}$).  However, the truncation threshold  $\lambda$ depends on the confidence level $\tau$. This can be inconvenient for the implementation of the algorithms. Some simple but coarser confidence bounds can be obtained by using a universal threshold independent of $\tau$, which is $\lambda=\max\{\sigma\sqrt{N},M\}+\upsilon \sigma$. In particular, we have the following result.
	
	\begin{theorem}\label{th:coset1sqrt}
		Let
		$\beta_i=\bar\beta\geq 2L$ for all $i=0,1,...$, and let
		$N\geq \upsilon^2$. Set
		$$\lambda=\max\left\{\sigma\sqrt{N}, M\right\}+\upsilon \sigma.$$
		\par\noindent
		Let $\wh{x}_N=N^{-1}\sum_{i=1}^{N}x_i$, where $x_i$ are the iterations of the RSMD algorithm
		defined by relations \rf{eq:md1} and \rf{t_grad0}.
		Then there is a random event $\cA_N\subset \Omega^{\otimes N}$ of probability at least $1-2e^{-\tau}$ such that for all $\omega^N\in \cA_N$ the following inequalities hold:
		\bse
		F(\wh{x}_N)- F_*&\leq &\frac{\bar\beta}{N}\sup_{z\in X}\varepsilon(x^N,z)\leq\\
		&\leq&\frac{C}{N}\left({\bar\beta R^2\Theta}+ \tau R\max\left\{\sigma\sqrt{N},M\right\}+{\tau \bar\beta^{-1}\max\{N\sigma^2,M^2\}}\right).
		\ese
		In particular, choosing $\bar\beta$ as in formula \rf{eq:betabar} we have
		\bqn
		F(\wh{x}_N)- F_*\leq \mbox{$\bar\beta\over N$}\sup_{z\in X}\varepsilon(x^N,z)\leq
		C\max\left\{{LR^2[\tau\vee\Theta]\over N},\;\tau\sigma R\sqrt{\Theta\over N}\right\}
		\eqn{eq:thm12sqr}
		for all $\omega^N\in \cA_N$.	
	\end{theorem}
	The values of the numerical constants $C$ in Theorem~\ref{th:coset1sqrt} can be obtained from the proof,
	cf. the bound in {\eqref{c'c''}}.
	
	%%%%%%%%%%%%%%
	
	%
	\section{Robust Confidence Bounds for Stochastic Optimization Methods}\label{sec:rcert}
	
	Consider an arbitrary algorithm for solving the problem \rf{prob1} based on $ N $ queries of the stochastic oracle. Assume that we have a sequence
	$\big(x_{i}, G(x_{i},\omega_{i+1})\big),\;i=0,...,N$, where $x_i\in X$ are the search points of some stochastic algorithm and $G(x_{i},\omega_{i+1})$ are the corresponding observations of the stochastic gradient. It is assumed that
	$x_{i}$ depends only on $\{(x_{j-1},\omega_{j}), j=1,\dots,i\}$.
	The approximate solution of the problem \rf{prob1} is defined in the form:
	$$
	\wh x_N=N^{-1}\sum_{i=1}^{N}x_{i}.
	$$
	Our goal is to construct a confidence interval with sub-Gaussian accuracy for $F(\wh x_N)-F_*$.
	To do this, we use the following fact. Note that for any
	$t\geq L$ the value
	\bqn
	~~~~~\epsilon_N(t)=N^{-1}\sup_{z\in X}\left\{\sum_{i=1}^{N}\big[\langle \nabla \phi(x_{i-1}),x_{i}-z\rangle+\psi(x_{i})-\psi(z)+ t V_{x_{i-1}}(x_i)\big]\right\}
	\eqn{eq:zazor}
	is an upper bound on the accuracy of the approximate solution $\wh x_N$:
	\bqn
	F(\wh x_N)-F_*\leq \epsilon_N(t)
	\eqn{eq:upperb}
	(see Lemma \ref{lem:smooth} in Appendix).
	This fact is true for any sequence of points $x_0,\dots, x_N$ in $X$, regardless of how they are obtained. However, since the function $\nabla \phi(\cdot)$ is not known, the estimate \rf{eq:upperb} cannot be used in practice. Replacing the gradients $\nabla \phi(x_{i-1})$ in \rf{eq:zazor} with their truncated estimates $ y_i $ defined in \rf{t_grad0} we get an implementable analogue of $ \epsilon_N(t) $:
	\bqn
	~~~~~\wh\epsilon_N(t)=N^{-1}\sup_{z\in X}\left\{\sum_{i=1}^{N}\big[\langle y_i,x_{i}-z\rangle+\psi(x_{i})-\psi(z)+ tV_{x_{i-1}}(x_i)\big]\right\}.
	\eqn{epsz}
	Note that computing $\wh\epsilon_N(t)$ reduces to solving a problem of the form \rf{prox} with $ \beta = 0 $. Thus, it is computationally not more complex than, for example, one step of the RSMD algorithm.
	Replacing $ \nabla \phi(x_{i-1}) $ with $ y_i $ introduces a random error.  In order to get a reliable upper bound for $ \epsilon_N(t)$, we need to compensate this error by slightly increasing $ \wh \epsilon_N(t) $. Specifically, we add to $ \wh\epsilon_N(t) $ the value
	\bse
	\bar\rho_N(\tau)&=&
	4R\sqrt{5\Theta\max\{N\sigma^2,M^2\tau\}}
	+16 R\max\{\sigma\sqrt{N\tau},M\tau\}+\nn
	&&+\min_{\mu\geq 0}\bigg\{ 20\mu\max\{N\sigma^2,M^2\tau\}+\mu^{-1}\sum_{i=1}^NV_{x_{i-1}}(x_i)\bigg\},
	\ese
	where $\tau>0$.
	
	\begin{proposition}\label{pro:certif}
		Let $\big(x_i,G(x_i, \omega_{i+1})\big)_{i=0}^N$ be the trajectory of a stochastic algorithm for which $ x_{i} $ depends only on $\{(x_{j-1},\omega_{j}), j=1,\dots,i\}$. Let $0<\tau\leq N/\upsilon^2$ and let $y_i=y_i(\tau)$ be truncated stochastic gradients defined in
		\rf{t_grad0}, where the threshold $\lambda=\lambda(\tau)$ is chosen in the form \rf{eq:lambda}. Then for any $t\ge L$ the value
		\[
		\Delta_N(\tau,t)=\wh\epsilon_N(t)+\bar\rho_N(\tau)/N
		\]
		is an upper bound for $\epsilon_N(t)$ with
		probability $1-2\mathrm{e}^{-\tau}$, so that
		\[
		\Prob\big\{F(\wh x_N)-F_*\leq \Delta_N(\tau,t)\big\}\ge 1-2\mathrm{e}^{-\tau}.
		\]
	\end{proposition}
	
	Since $\Delta_N(\tau,t)$ monotonically increases in $t$ it suffices to use this bound for $t = L$  when $L$ is known. Note that, although $\Delta_N(\tau,t)$ gives an upper bound for $ \epsilon_N(t) $, Proposition \ref{pro:certif} does not guarantee that $\Delta_N(\tau,t)$ is sufficiently close to $ \epsilon_N(t)$. However, this property holds for the RSMD algorithm with a constant step, as follows from the next result.
	
	\begin{corollary}\label{cor:smdb} Under the conditions of Proposition~\ref{pro:certif}, let the vectors $ x_0, \dots, x_N $ be given by the  RSMD recursion \rf{eq:md1}--\rf{t_grad0}, where $ \beta_i = \bar \beta \ge 2L $, $ i = 0, ..., N-1 $. Then
		\bqn
		\bar\rho_N(\tau)&\leq &N\epsilon_N(\bar\beta)+4R\sqrt{5\Theta\max\{N\sigma^2,M^2\tau\}}+\nn
		&&+16 R\max\{\sigma\sqrt{N\tau},M\tau\}+20\bar\beta^{-1}\max\{N\sigma^2,M^2\tau\}.
		\eqn{barrho}
		Moreover,  if $\bar\beta\geq \max\left\{2L,{\sigma\sqrt{N}\over R\sqrt{\Theta}}\right\}$ then
		\[
		\bar\rho_N(\tau)\leq N\epsilon_N(\bar\beta)+C_3LR^2[\Theta\vee\tau]+{C_4\sigma R\sqrt{N[\Theta\vee \tau]}},
		\]
		and with probability at least $1-4\mathrm{e}^{-\tau}$ the value $\Delta_N(\tau,\bar\beta)$ satisfies the inequalities
		\bqn
		\epsilon_N(\bar\beta)\leq \Delta_N(\tau,\bar\beta)\leq 3\epsilon_N(\bar\beta)+2C_3{LR^2[\Theta\vee\tau]\over N}+2C_4{\sigma R\sqrt{[\Theta\vee \tau]\over N}},
		\eqn{eq:jnm}
		where $C_3>0$ and $C_4>0$ are numerical constants.	
	\end{corollary}
	
	The values of the numerical constants $ C_3 $ and $ C_4 $ can be derived from the proof of this corollary.
	
	%%%%%%%%%%%%%%%%%%%%
	
	\section{Robust Confidence Bounds for Quadratic Growth Problems}
	\label{sec:strong}

	In this section, it is assumed that $ F $ is a  function with quadratic growth on $ X $ in the following sense (cf. \cite{necoara2018linear}). Let $ F $ be a continuous function on $ X $ and let $X_*\subset X$ be the set of its minimizers on $ X $.
	Then $ F $ is called a {\em function with quadratic growth  on $ X $} if there is a constant $ \kappa> 0 $ such that for any $x\in X$ there exists
	$\bar{x}(x)\in X_*$ such that the following inequality holds:
	\bqn
	F(x)-F_*\geq {\kappa\over 2}\|x-\bar{x}(x)\|^2.
	\eqn{eq:qmin}
	
	Note that every strongly convex function  $F$ on $ X $  with the strong convexity coefficient $ \kappa $ is a function with quadratic growth on $ X $.
	However, the assumption of strong convexity, when used together with the Lipschitz condition with constant $L$ on the gradient of $F$, has the disadvantage that, except for the case when $\|\cdot\|$ is the Euclidean norm, the ratio $ L/\kappa $ depends on the dimension of the space $E$.
	%--
	For example, in the important cases where $\|\cdot\|$ is the $\ell_1$-norm, the nuclear norm, the total variation norm, etc., one can easily check (cf. \cite{juditsky2014deterministic}) that there are no functions with Lipschitz continuous  gradient such that the ratio $L/\kappa$ is smaller than the dimension of the space.
	Replacing the strong convexity with the growth condition \rf{eq:qmin}  eliminates this problem, see the examples in \cite{necoara2018linear}.
	On the other hand,
	assumption \rf{eq:qmin} is quite natural in the composite optimization problem since in many interesting examples the function $\phi$ is smooth and the non-smooth part $\psi$ of the objective function is strongly convex.
	%--
	In particular, if $ E = \bR^n $ and the norm is  the $\ell_1$-norm, this allows us to consider such strongly convex components as
	the negative entropy $\psi(x)=\kappa \sum_j x_j\ln x_j$
	(if $X$ is standard simplex in $\bR^n$),
	$\psi(x)=
	\gamma(\kappa)
	\|x\|_p^p$ with  $1\leq p\leq 2$
	and with the corresponding choice of $ \gamma(\kappa)>0$
	(if $X$ is a convex compact in $\bR^n$) and others.
	In all these cases,
	condition  \rf{eq:qmin} is fulfilled with a known constant $\kappa$, which
	allows for the use of the
	approach of \cite{juditsky2014deterministic,OML2011%%juditsky2011_1%
	} to improve the confidence bounds of the stochastic mirror descent.
	
	\par
	The RSMD algorithm for quadratically growing functions will be defined in stages.
	At each stage, for specially selected $ r> 0 $ and $ y \in X $ it solves  an auxiliary problem
	\[
	\min_{x\in X_{r}(y)} F(x)
	\]
	using the RSMD. Here
	\[
	X_r(y)=\{x\in X:\;\|x-y\|\leq r\}.
	\]
	We initialize the algorithm by choosing arbitrary $y_0=x_0\in X$ and $r_0\geq \max_{z\in X}\|z-x_0\|$. We set $r^2_k=2^{-k} r^2_0$, $k=1,2,...$. Let $ C_1 $ and $ C_2 $ be the numerical constants in the bound~\rf{eq:thm12} of Theorem~\ref{th:coset1}. For a given parameter $ \tau> 0 $, and $ k = 1,2, \dots $ we define the values
	\bqn
	\ovN_k=\max\left\{4C_1{L[\tau\vee\Theta]\over \kappa},\,16C_2{\sigma^2[\tau\vee\Theta]\over \kappa^2r^2_{k-1}}\right\},\;\;N_k=\rfloor \ovN_k\lfloor.
	\eqn{N_k}
	Here $\rfloor t\lfloor$ denotes the smallest integer greater than or equal to $t$. Set
	\[
	m(N):=\max\left\{k:\;\sum_{j=1}^k N_j\leq N\right\}.
	\]
	Now, let $k\in\{1,2,\dots,m(N)\}$.
	At the  $k$-th stage of the algorithm,
	we solve the problem of minimization of $F$ on the ball $X_{r_{k-1}}(y_{k-1})$, we find
	its approximate solution $\wh{x}_{N_k}$ according to
	\rf{eq:md1}--\rf{eq:asol}, where we replace
	$ x_0$ by $ y_{k-1} $, $X$ by $X_{r_{k-1}}(y_{k-1})$,
	$R$ by $r_{k-1}$, $N$ by $N_k$,  and set
	\[
	\lambda=\max\left\{\sigma\sqrt{N\over \tau}, Lr_{k-1}\right\}+\upsilon \sigma,
	\]
	and
	\[
	\beta_i \equiv \max\left\{2L, \,{\sigma\sqrt{N}\over r_{k-1}\sqrt{\Theta}}\right\}.
	\]
	It is assumed that, at each stage $k$ of the algorithm, an exact solution of the minimization problem
	$$
	\min_{z\in X_{r_{k-1}}(y_{k-1})}\left\{\langle a,z\rangle + \psi(z) + \beta \vartheta (z) \right\}
	$$
	is available for any $a\in E$ and $\beta>0$.
	At the output of the $k$-th stage of the algorithm, we obtain $y_k:=\wh{x}_{N_k}$.
	\begin{theorem}\label{th:strc}
		Assume that $ m(N) \geq 1 $, i.e. at least one stage of the algorithm described above is completed.
		Then there is a random event $\cB_N\subset\Omega^{\otimes N}$ of probability at least $1- 2m(N)e^{-\tau}$ such that for $\omega^N \in \cB_N$ the approximate solution $ y_{m(N)} $ after  $ m(N) $ stages of the algorithm satisfies the inequality
		\bqn \quad \qquad
		F(y_{m(N)})-F_*\leq C
		\max\left\{\kappa r_0^2\,
		{2^{-N/4}}
		,\;\kappa r_0^2\exp\left(-{C'
			\kappa N\over L[\tau\vee \Theta]}\right),\;{\sigma^2[\tau\vee \Theta]\over \kappa N}
		\right\}.
		\eqn{th2bound}
	\end{theorem}
	
	Theorem \ref{th:strc} shows that, for  functions with quadratic growth, the deterministic error component can be significantly reduced
	-- it becomes exponentially decreasing in $N$. The stochastic error component is also significantly reduced. Note that the  factor $m(N)$ is of logarithmic order and has little effect on the probability of deviations. Indeed, it follows  from  \rf{N_k} that
	$m(N)\le C \ln \left(\frac{C' \kappa^2 r_0^2 N}{\sigma^2(\tau\vee\Theta)}\right)$.
	Neglecting this factor in the probability of deviations and considering the stochastic component of the error, we see that the confidence bound of Theorem~\ref{th:strc} is approximately sub-exponential rather than sub-Gaussian.
	
	\section{Conclusion}

	We have considered algorithms of smooth stochastic optimization when the  distribution of noise in observations
	has heavy tails. It is shown that by truncating the observed gradients with a suitable threshold one can construct confidence sets for the approximate solutions that are similar to those in the case of ``light tails''. It should be noted that the order of the deterministic error in the obtained bounds is suboptimal --- it is substantially greater than the optimal rates achieved by the accelerated algorithms
	\cite{ghadimi2012optimal,lan2012optimal}, namely, $O(LR^2N^{-2})$ in the case of convex objective function and $O(\exp(-N\sqrt{\kappa/L}))$ in the strongly convex case.
	On the other hand, the proposed approach cannot be used to obtain robust versions of the accelerated algorithms since applying it to such algorithms leads to accumulation of the bias caused by the truncation of the gradients.
	The problem of constructing accelerated robust stochastic algorithms with optimal guarantees remains open.
	%%%%%%%%%%%%%%%%

	%%%%%%%%%%%%%%%%

\appendix{\hfill{\em APPENDIX}
	
	\textbf{A.1.~Preliminary remarks.}
	We start with the following known result.
	\begin{lemma}\label{lem:smooth}
		Assume that $ \phi $ and $ \psi $ satisfy the assumptions of Section~\ref{sec:pstat}, and let $ x_0, \dots, x_N $ be some points of the set $X$.
		Define
		\[
		\varepsilon_{i+1}(z):=\langle \nabla \phi(x_i),x_{i+1}-z\rangle+\langle \psi'(x_{i+1}),x_{i+1}-z\rangle+LV_{x_i}(x_{i+1}).
		\]
		Then for any $z\in X$ the following inequality holds:
		\[
		F(x_{i+1})-F(z)\leq \varepsilon_{i+1}(z).
		\]
		In addition, for
		$\widehat{x}_N={1\over N}\sum_{i=1}^{N} x_i$ we have
		\[
		F(\widehat{x}_N)-F(z)\leq N^{-1}\sum_{i=1}^{N}[ F(x_i)-F(z)]\leq N^{-1}\sum_{i=0}^{N-1}\varepsilon_{i+1}(z).
		\]\end{lemma}
	\begin{proof} Using the property $V_{x}(z)\geq \half \|x-z\|^2$, the convexity of functions $\phi$ and $\psi$ and the Lipschitz condition on $\nabla\phi$ we get that, for any $z\in X$,
		\bse
		F(x_{i+1})-F(z)&=&[\phi(x_{i+1})-\phi(z)]+[\psi(x_{i+1})-\psi(z)]=\nn
		&=&[\phi(x_{i+1})-\phi(x_{i})]
		+[\phi(x_{i})-\phi(z)]+[\psi(x_{i+1})-\psi(z)]\leq\nn
		&\leq& [\langle \nabla \phi(x_i),x_{i+1}-x_i\rangle+LV_{x_i}(x_{i+1})]+\langle \nabla \phi(x_i),x_{i}-z\rangle
		+\psi(x_{i+1})-\psi(z)\leq\nn
		&\leq &
		\langle \nabla \phi(x_i),x_{i+1}-z\rangle+\langle \psi'(x_{i+1}),x_{i+1}-z\rangle+LV_{x_i}(x_{i+1})=\varepsilon_{i+1}(z).
		\ese %{eq:lipc}
		Summing up over $ i $ from 0 to $ N-1 $ and using the convexity of $ F $ we obtain the second result of the lemma. \qed
	\end{proof}
	
	In what follows, we denote by $\bE_{x_{i}}\{\cdot\}$ the conditional expectation for fixed $x_{i}$.
	
	\begin{lemma}\label{lem:cut}
		Let the assumptions of Section \ref{sec:pstat} be fulfilled and let $ x_i $ and $ y_i $ satisfy the RSMD recursion, cf. \rf{eq:md1} and \rf{t_grad0}. Then
		\begin{equation*}
		(a)\quad\quad\;\,\|\xi_i\|_*\leq  2(M+\upsilon\sigma)+\lambda,\quad\quad\quad\quad\quad{}
		\end{equation*}
		\begin{equation}\label{eq:lemxi}
		\quad\;		(b)\quad\quad\quad\|\bE_{x_{i-1}}\{\xi_i\}\|_*\leq (M+\upsilon\sigma)\left({\sigma\over \lambda}\right)^{2}+{\sigma^2\over \lambda},
		\end{equation}
		\begin{equation*}
		(c)\quad\quad\quad
		\left(\bE_{x_{i-1}}\{\|\xi_i\|_*^2\}\right)^{1/2}\leq \sigma+(M+\upsilon\sigma){\sigma\over \lambda}.
		\end{equation*}
	\end{lemma}
	
	%%%%%%%%%%%%%%%%%%%%%%%%%%%%%%%%%%%%%%%%%%%%%%%%%%%%
	
	\begin{proof}
		Set $\chi_i=1_{\|G(x_{i-1},\omega_i)-g(\bar{x})\|_*> L\|x_{i-1}-\bar{x}\|+\lambda+\upsilon\sigma}$.
		Note that by construction
		\\$\chi_i\leq \eta_i:= 1_{\|G(x_{i-1},\omega_i)-\nabla f(x_{i-1}\|_*> \lambda}$.
		We have
		\bse
		\xi_i&=&y_i-\nabla \phi(x_{i-1})=[G(x_{i-1},\omega_i)-\nabla \phi(x_{i-1})](1-\chi_i)+[g(\bar x)-\nabla \phi(x_{i-1})]\chi_i=\\
		&=&[G(x_{i-1},\omega_i)-g(\bar{x})](1-\chi_i)+[g(\bar{x})-\nabla \phi(x_{i-1})]=\\
		&=&[G(x_{i-1},\omega_i)-\nabla \phi(x_{i-1})]+[g(\bar{x})-G(x_{i-1},\omega_i)]\chi_i.
		\ese
		Therefore,
		\bse
		\|\xi_i\|_*\leq \|[G(x_{i-1},\omega_i)-g(\bar{x})](1-\chi_i)\|_*+\|g(\bar{x})-\nabla \phi(x_{i-1})\|_*\leq 2(M+\upsilon\sigma)+\lambda.
		\ese
		%\marginpar{\fbox{Check!}~}
		Moreover, since $\bE_{x_{i-1}}\{G(x_{i-1},\omega_i)\}=\nabla \phi(x_{i-1})$ we have
		\bse
		\|\bE_{x_{i-1}}\{\xi_i\}\|_*&=&\big\|\bE_{x_{i-1}}\{[(G(x_{i-1},\omega_i)-\nabla \phi(x_{i-1}))-(g(\bar{x})-\nabla \phi(x_{i-1}))]\chi_i\}\big\|_*\leq\\
		&\leq& \bE_{x_{i-1}}\{[\|G(x_{i-1},\omega_i)-\nabla \phi(x_{i-1})\|_*+\|g(\bar{x})-\nabla \phi(x_{i-1})\|_*]\chi_i\}\leq\\
		&\leq&
		\bE_{x_{i-1}}\{\|G(x_{i-1},\omega_i)-\nabla \phi(x_{i-1})\|_*\eta_i\big\}+(M+\upsilon\sigma)
		\bE_{x_{i-1}}\{\eta_i\}\leq\\
		&\leq& {\sigma^{2}\over \lambda}+(M+\upsilon\sigma)\left({\sigma\over \lambda}\right)^{2}.
		\ese
		Further,
		\[
		\|\xi_i\|_* \leq \|G(x_{i-1},\omega_i)-\nabla \phi(x_{i-1})\|_*(1-\chi)+\|g(\bar{x})-\nabla \phi(x_{i-1})\|_*\chi_i,
		\]
		and
		\bse
		\bE_{x_{i-1}}\{\|\xi_i\|_* ^2\}^{1/2}&\leq& \bE_{x_{i-1}}\{\|G(x_{i-1},\omega_i)-\nabla \phi(x_{i-1})\|_*^2\}^{1/2}+\\
		&&{}\quad\quad\quad\quad\quad\quad\quad
		+\,\bE_{x_{i-1}}\{\|g(\bar{x})-\nabla \phi(x_{i-1})\|^2_*\chi_i\}^{1/2}\leq\\
		&\leq& \sigma+(M+\upsilon\sigma)\bE\{\chi_i\}^{1/2}\leq \sigma+(M+\upsilon\sigma)\bE_{x_{i-1}}\{\eta_i\}^{1/2}\leq\\ &\leq&\sigma+(M+\upsilon\sigma){\sigma\over \lambda}.
		\ese\qed
	\end{proof}
	
	The following lemma gives bounds for the deviations of the sums $\sum_{i}\la \xi_{i},x_{i-1}-z\ra$ and $\sum_{i}\|\xi_i\|_*^2$.
	\begin{lemma}\label{le:mart1} Let the assumptions of Section~\ref{sec:pstat} be fulfilled and let $ x_i $ and $ y_i $ satisfy the recursion of RSMD, cf. \rf{eq:md1} and \rf{t_grad0}.
		\item{(i)} If $\tau\leq {N/\upsilon^2}$ and $\lambda=\max\left\{\sigma\sqrt{N\over \tau}, M\right\}+\upsilon \sigma$ then, for any $z\in X$,
		\bqn
		\Prob\left\{\sum_{i=1}^N\la \xi_{i},z-x_{i-1}\ra\geq 16R\max\{\sigma\sqrt{N\tau},M\tau\}\right\}\leq e^{-\tau},
		\eqn{eq:zbound}
		and
		\bqn
		\Prob\left\{\sum_{i=1}^N\|\xi_i\|_*^2\geq 40\max\{N\sigma^2,M^2\tau\}\right\}\leq e^{-\tau}.
		\eqn{eq:sbound}
		\item{(ii)} If  $N\geq \upsilon^2$ and $\lambda=\max\left\{\sigma\sqrt{N}, M\right\}+\upsilon \sigma$ then, for any $z\in X$,
		\bqn
		\Prob\left\{\sum_{i=1}^N\la \xi_{i},z-x_{i-1}\ra\geq 8(1+\tau)R\max\{\sigma\sqrt{N},M\}\right\}\leq e^{-\tau},
		\eqn{eq:zboundsqrt}
		and
		\bqn
		\Prob\left\{\sum_{i=1}^N\|\xi_i\|_*^2\geq 8(2+3\tau)\max\{N\sigma^2,M^2\}\right\}\leq e^{-\tau}.
		\eqn{eq:sboundsqrt}
	\end{lemma}
	
	\begin{proof}
		Set $\zeta_{i}= \la \xi_{i},z-x_{i-1}\ra$ and $\varsigma_i=\|\xi_i\|_*^2$, $i=1,2,\dots $
		Using Lemma~\ref{lem:cut} it is easy to check that the following inequalities are fulfilled
		\bqn
		\begin{array}{lrcl}
			(a)&\quad |\bE_{x_{i-1}} \{\zeta_i\}|&\leq& D\left[(M+\upsilon \sigma)(\sigma/\lambda)^2+\sigma^2/\lambda\right],\\
			(b)&\quad|\zeta_i|&\leq& D[2(M+\upsilon \sigma)+\lambda],\\
			(c)&\quad(\bE_{x_{i-1}} \{\zeta_i^2\})^{1/2}&\leq &D[\sigma+(M+\upsilon \sigma)\sigma/\lambda]\\
		\end{array}
		\eqn{eq:zeta1}
		and
		\bqn
		\begin{array}{lrcl}
			(a)\quad& \bE_{x_{i-1}}\{ \varsigma_i\}&\leq&[\sigma+(M+\upsilon \sigma)\sigma/\lambda]^2,\\
			(b)\quad&\varsigma_i&\leq&  [2(M+\upsilon\sigma)+\lambda]^2,\\
			(c)\quad&(\bE_{x_{i-1}}\{ \varsigma_i^2\})^{1/2}&\leq &[\sigma+(M+\upsilon \sigma)\sigma/\lambda]\,[2(M+\upsilon\sigma)+\lambda].
		\end{array}
		\eqn{eq:varsigma1}
		\par
		In what follows, we apply several times the Bernstein inequality, and each time we will use
		the same notation $r$, $A$, $s$ for the  values that are, respectively, the uniform upper bound of the expectation, the maximum absolute value, and the standard deviation of a random variable.
		\par
		\textbf{1$^o$.}
		We first prove the statement $(i)$.
		We start with the case $M\leq \sigma\sqrt{N\over \tau}$.  It follows from \rf{eq:zeta1} that in this case
		\bqn
		\begin{array}{rcl}
			|\bE_{x_{i-1}} \{\zeta_i\}|&\leq& 2D\sigma^2/\lambda \leq 4R\sigma\sqrt{\tau\over N}=:r,\\
			|\zeta_i|&\leq & A:=3\lambda D\leq 6R\lambda,\\
			(\bE_{x_{i-1}} \{\zeta_i^2\})^{1/2}&\leq &s:=2D\sigma\leq 4R\sigma.
		\end{array}
		\eqn{eq:zeta2}
		Using \rf{eq:zeta2} and Bernstein's inequality for martingales (see, for example, \cite{freedman1975tail}) we get
		\bse
		\Prob\left\{\sum_{i=1}^N\zeta_i\geq 16R\sigma\sqrt{N\tau}\right\}&\leq& \Prob\left\{\sum_{i=1}^N\zeta_i\geq Nr+3s \sqrt{N\tau}\right\}\leq
		\\&\leq&
		\exp\left\{-{9\tau\over 2+{2\over 3}{3\sqrt{\tau}A\over s\sqrt{N}}}\right\}
		\leq
		\\&\leq&
		\exp\left\{-{9\tau\over 2+	3\big(1+\upsilon\sqrt{\tau/N}\big)}
		\right\}\leq
		e^{-\tau}
		\ese
		for all $\tau>0$ satisfying the condition $\tau\leq {16N/(9\upsilon^2)}$.
		On the other hand, in the case under consideration, the following inequalities hold (cf. \rf{eq:varsigma1} and \rf{eq:zeta2})
		\[
		\begin{array}{rcl}
		\bE_{x_{i-1}} \{\varsigma_i\}\leq \underbrace{4\sigma^2}_{=:r},\;\;
		\varsigma_i\leq \underbrace{9\lambda^2}_{=:A},\;\;
		(\bE_{x_{i-1}} \{\varsigma_i^2\})^{1/2}\leq \underbrace{6\lambda\sigma}_{=:s}.
		\end{array}
		\]
		Thus,
		\[
		Nr+3s\sqrt{\tau N}=4N\sigma^2+18\lambda \sigma\sqrt{\tau N}=22N\sigma^2+18\upsilon \sigma^2\sqrt{N\tau}\leq 40N\sigma^2
		\]
		for $0<\tau\leq {N/\upsilon^2}$. Applying again the Bernstein inequality, we get
		\bse
		\Prob\left\{\sum_{i=1}^N\varsigma_i\geq 40N\sigma^2
		\right\}\leq \exp\left\{-{9\tau\over 2+	\big(3+3\upsilon\sqrt{\tau/N}\big)}
		\right\}\leq e^{-\tau}
		\ese
		for all $\tau>0$ satisfying the condition $\tau\leq N/\upsilon^2$.
		\par
		\textbf{2$^o$.} Assume now that $M> \sigma\sqrt{N\over \tau}$, so that $\lambda=M+\upsilon \sigma$ and $\sigma^2\leq {M^2\tau/N}$.
		Then
		\[
		|\bE_{x_{i-1}}\zeta_i|\leq 4R\sigma^2/\lambda\leq \underbrace{4RM\tau/N}_{=:r},\;\;
		|\zeta_i|\leq R(2(M+\upsilon\sigma)+\lambda)=\underbrace{6R(M+\upsilon\sigma)}_{=:A},
		\]
		\[(\bE_{x_{i-1}} \{\zeta_i^2\})^{1/2}\leq 4R\sigma\leq \underbrace{4MR\sqrt{\tau/N}}_{=:s}.
		\]
		Further,
		\[
		Nr+3s\sqrt{\tau N}=4RM\tau+12RM\tau=16RM\tau,
		\]
		and applying again the Bernstein inequality we get
		\bse
		\Prob\left\{\sum_{i=1}^N\zeta_i\geq 16RM\tau\right\}&\leq&
		\exp\left\{-{9\tau\over 2+{2\over 3}{3\sqrt{\tau}A\over s\sqrt{N}}}\right\}\leq
		\exp\left\{-{9\tau\over 2+	\big(3+3\upsilon\sigma/M\big)}\right\}\leq
		\\&\leq& \exp\left\{-{9\tau\over 5+3\upsilon\sqrt{\tau/N}}
		\right\}\leq e^{-\tau}
		\ese
		for all $\tau>0$, satisfying the condition $\tau\leq 16N/(9\upsilon^2)$.
		Next, in this case
		\[
		\begin{array}{rcl}
		\bE_{x_{i-1}} \{\varsigma_i\}\leq 4\sigma^2\leq \underbrace{4\tau M^2/N}_{=:r},\;\;
		\varsigma_i\leq \underbrace{9\lambda^2}_{=:A},\;\;
		(\bE_{x_{i-1}} \{\varsigma_i^2\})^{1/2}\leq \underbrace{6\lambda\sigma}_{=:s}\leq 6\lambda M\sqrt{\tau/N}.
		\end{array}
		\]
		Now,
		\[
		Nr+3s\sqrt{\tau N}=4\tau M^2+18\lambda\sigma \sqrt{\tau N}\leq 22M^2\tau +18\upsilon \sigma^2\sqrt{N\tau}\leq 40M^2\tau,
		\]
		for $\tau\leq {N/\upsilon^2}$. Applying once again the Bernstein inequality we get
		\bse
		\Prob\left\{\sum_{i=1}^N\varsigma_i\geq 40\tau M^2
		\right\}\leq \exp\left\{-{9\tau\over 2+
			\big(3+3\upsilon\sqrt{\tau/N}\big)}
		\right\}\leq e^{-\tau},
		\ese
		for all $\tau>0$ satisfying the condition  $\tau\leq N/\upsilon^2$.
		\par \textbf{3$^o$.} Now, consider the case $\lambda=\max\{M,\sigma\sqrt{N}\}+\sigma\upsilon$. Let $M\leq \sigma\sqrt{N}$, so that $\lambda=\sigma(\sqrt{N}+\upsilon)$. We argue in the same way as in the proof of $(i)$. By virtue of \rf{eq:zeta1} we have
		\[
		\begin{array}{rcl}
		|\bE_{x_{i-1}} \{\zeta_i\}|&\leq&4R{\sigma\over \sqrt{N}}=: r,\\
		(\bE_{x_{i-1}} \{\zeta_i^2\})^{1/2}&\leq & 4R\sigma=:s,\\
		|\zeta_i|&\leq &  6R\lambda\leq 12R\sigma\sqrt{N}=3s\sqrt{N}.
		\end{array}
		\]
		Hence, using the Bernstein inequality we get
		\bse
		&&	\Prob\left\{\sum_{i=1}^N\zeta_i\geq 8R\sigma\sqrt{N} (\tau+1)\right\}
		\leq \Prob\left\{\sum_{i=1}^N\zeta_i\geq Nr+(2\tau+1)s\sqrt{N}\right\}
		\leq
		\\&&\leq
		\exp\left\{-{(2\tau+1)^2s^2N\over 2s^2N+{2\over 3}3s^2{N}(2\tau+1)}\right\}\leq
		\exp\left\{-{(2\tau+1)^2\over 2+2(2\tau+1)}
		\right\}\leq e^{-\tau}.
		\ese
		From \rf{eq:varsigma1} we also have
		\[
		\begin{array}{rcl}
		\bE_{x_{i-1}} \{\varsigma_i\}&\leq& \underbrace{4\sigma^2}_{=:r},\\
		(\bE_{x_{i-1}} \{\varsigma_i^2\})^{1/2}&\leq& 6\lambda\sigma\leq 12\sigma^2\sqrt{N}=:s,\\
		\varsigma_i&\leq& 9\lambda^2\leq 36\sigma^2N=4s\sqrt{N}.
		\end{array}
		\]
		Now,  applying again the Bernstein inequality we get
		\bse\Prob\left\{\sum_{i=1}^N\varsigma_i\geq 16N\sigma^2 +24N\sigma^2\tau
		\right\}&=&\Prob\left\{\sum_{i=1}^N\varsigma_i\geq Nr+(2\tau+1)s\sqrt{N}
		\right\}\leq\\
		&\leq& \exp\left\{-{(2\tau+1)s^2N\over [2+2(2\tau+1)]s^2N}
		\right\}\leq e^{-\tau}.
		\ese
		Proofs of the bounds \rf{eq:zboundsqrt} and \rf{eq:sboundsqrt} in the case $M>\sigma\sqrt{N}$ and $\lambda=M+\sigma \upsilon$ follow the same lines.
		\qed
	\end{proof}
	
	\textbf{A.2.~Proof of Proposition~\ref{pr:muprop}.}
	We first prove  inequality \rf{simple0}. In view of \rf{prox},
	the optimality condition for  \rf{eq:md1} %with $\xi=y_{i+1}$, $x=x_i$ and  $\beta=\beta_i$
	has the form
	\[
	\langle y_{i+1}+\psi'(x_{i+1})+\beta_i[\vartheta'(x_{i+1})-\vartheta'(x_i)],z-x_{i+1}\rangle\ge 0,\;\;\forall \; z\in X,
	\]
	or, equivalently,
	\bse
	\langle y_i+\psi'(x_{i+1}),x_{i+1}-z\rangle&\leq& \beta_i\langle [\vartheta'(x_{i+1})-\vartheta'(x_i)],z-x_{i+1}\rangle=
	\langle \beta_iV'_{x_i}(x_{i+1}),z-x_{i+1}\rangle=\\
	&=&\beta_i[V_{x_i}(z)-V_{x_{i+1}}(z)-V_{x_i}(x_{i+1})],\;\;\forall \; z\in X,
	\ese
	where the last equality follows from the following remarkable identity (see, for example, \cite{chen1993convergence}): for any $u, u'$ and $w\in X$
	\[
	\langle V'_{u}(u'),w-u'\rangle=V_u(w)-V_{u'}(w)-V_u(u').
	\]
	Since, by definition, $ \xi_i =y_i-\nabla\phi(x_{i-1}) $ we get
	\bqn
	{}\quad\quad\quad\quad\langle \nabla \phi (x_i),x_{i+1}-z\rangle+\langle\psi'(x_{i+1}),x_{i+1}-z\rangle
	&\leq&
	\beta_i[V_{x_i}(z)-V_{x_{i+1}}(z)-V_{x_i}(x_{i+1})]-
	\nn
	&&-\langle \xi_{i+1},x_{i+1}-z\rangle.
	\eqn{eq:BT}
	It follows from Lemma~\ref{lem:smooth} and the condition $\beta_i\geq 2L$  that
	\bse
	F(x_{i+1})-F(z)\leq \varepsilon_{i+1}(z)\leq
	\langle \nabla \phi(x_i),x_{i+1}-z\rangle+\langle \psi'(x_{i+1}),x_{i+1}-z\rangle+{\beta_i\over 2}V_{x_i}(x_{i+1}).
	\ese
	Together with
	\eqref{eq:BT}, this inequality implies
	\bse
	\varepsilon_{i+1}(z)\leq \beta_i[V_{x_i}(z)-V_{x_{i+1}}(z)-\half V_{x_i}(x_{i+1})]-\langle \xi_{i+1},x_{i+1}-z\rangle.
	\ese
	On the other hand, due to the strong convexity of $V_x(\cdot)$ we have
	\bse
	\langle\xi_{i+1},z-x_{i+1}\rangle-{\beta_i\over 2}V_{x_i}(x_{i+1})&=&\langle\xi_i,z-x_{i}\rangle+\langle\xi_{i+1},x_{i}-x_{i+1}\rangle-{\beta_i\over 2}V_{x_i}(x_{i+1})\nn
	&\le &\langle\xi_{i+1},z-x_{i}\rangle+{\|\xi_{i+1}\|^2_*\over \beta_i}.
	\ese
	Combining these inequalities, we obtain
	\bqn{}\quad\quad\quad\quad
	F(x_{i+1})-F(z)\leq\varepsilon_{i+1}(z)\leq
	\beta_i[V_{x_i}(z)-V_{x_{i+1}}(z)]-\langle \xi_{i+1},x_{i}-z\rangle+{\|\xi_{i+1}\|^2_*\over \beta_i} \ \quad
	\eqn{almost}
	for all $z\in X$. Dividing \rf{almost}
	by $ \beta_i $ and taking the sum over $ i $ from $0$ to $ N-1 $  we obtain \rf{simple0}.
	
	We now prove the bound \rf{simple1}. Applying Lemma~6.1 of \cite{nemirovski2009robust} with $ z_0 = x_0 $ we get
	\bqn
	\forall z\in X,\quad\quad\sum_{i=1}^N\beta_{i-1}^{-1} \la \xi_{i},z-z_{i-1}\rangle\le V_{x_0}(z)+\half \sum_{i=1}^N \beta_{i-1}^{-2}\|\xi_i\|_*^2,
	\eqn{lemma61}
	where $z_i=\argmin_{z\in X}\big\{ \mu_{i-1}\langle \xi_i,z\rangle + V_{z_{i-1}}(z)\big\}$ depends only on $z_0,\xi_1,\dots,\xi_{i}$. Further,
	\bse
	\sum_{i=1}^{N}\beta^{-1}_{i-1}\la \xi_{i},z-x_{i-1}\ra&=&\sum_{i=1}^{N}\beta^{-1}_{i-1}[\la \xi_{i},z_{i-1}-x_{i-1}\ra+\la \xi_{i},z-z_{i-1}\ra]\leq
	\\&\leq& V_{x_0}(z)+\sum_{i=1}^N\beta^{-1}_{i-1} \la \xi_{i},z_{i-1}-x_{i-1}\ra+\half\beta^{-2}_{i-1} \|\xi_i\|_*^2.
	\ese
	Combining this inequality with \rf{simple0}, we get
	\rf{simple1}.
	\qed
	
	\textbf{A.3.~Proof of Corollary~\ref{cor:secmom}.}
	Note that \rf{eq:cor1-1} is an immediate consequence of \rf{simple0} and of the bounds for the moments of $ \| \xi_i \|_* $ given in Lemma~\ref{lem:cut}.
	Indeed, (\ref{eq:lemxi})(b) implies that, under the conditions of  Corollary~\ref{cor:secmom},
	\[
	|\bE_{x_{i-1}}\{\langle \xi_i,z-x_{i-1}\ra \}|\leq 2D\left[(M+\upsilon\sigma)\left({\sigma\over \lambda}\right)^{2}+{\sigma^2\over \lambda}\right]\leq {2D\sigma^2\over \lambda}\leq {2D\sigma\over \sqrt{N}}.
	\]
	Further, due to (\ref{eq:lemxi})(c),
	\[
	\bE_{x_{i-1}}\{\|\xi_i\|_*^2\}^{1/2}\leq \sigma+(M+\upsilon\sigma){\sigma\over \lambda}\leq 2\sigma.
	\]
	Taking the expectation of both sides of \rf{simple0} and using the last two inequalities we get \rf{eq:cor1-1}.
	The bound \rf{eq:cor1supz} is proved in a similar way, with the only difference that instead of inequality \rf{simple0} we use \rf{simple1}.
	\qed
	
	\textbf{A.4.~Proof of Theorem~\ref{th:coset1}.}
	By virtue of part (i) of Lemma~\ref{le:mart1}, under the condition $\tau\leq N/\upsilon^2$ we have that, with probability of at least
	$1-2e^{-\tau}$,
	\[\begin{array}{rcl}
	\sum_{i=1}^N\la \xi_{i},z_{i-1}-x_{i-1}\ra&\leq& 16R\max\{\sigma\sqrt{N\tau},M\tau\},\\
	\sum_{i=1}^N\|\xi_i\|_*^2&\leq& 40\max\{N\sigma^2,M^2\tau\}.
	\end{array}
	\]
	Plugging these bounds in \rf{simple1} we obtain that, with probability at least $1-2e^{-\tau}$, the following holds:
	\bse
	\bar\beta\sup_{z\in X}\varepsilon(x^N,z)&\le&2\bar\beta V_{x_0}(z)+\sum_{i=1}^{N}\Big[
	\langle \xi_{i},z_{i-1}-x_{i-1}\rangle +{\mbox{$3\over 2$}\bar\beta^{-1}\|\xi_{i}\|_*^2}\Big]\leq\\
	&\leq &2\bar\beta R^2\Theta+16R\max\{\sigma\sqrt{N\tau},M\tau\}+{60\bar\beta^{-1}\max\{N\sigma^2,M^2\tau\}}.
	\ese
	Next, taking $\bar\beta=\max\big\{2L, {\sigma\over R}\sqrt{N\over \Theta}\big\}$ we get
	\bqn
	N[F(\widehat{x}_N)-F(z)]&\le&
	\max\{4LR^2\Theta, 2\sigma R\sqrt{N\Theta}\}+16R\max\{\sigma\sqrt{N\tau},M\tau\}+\nn
	&&+\,60\max\{LR^2\tau /2, \sigma R\sqrt{N\Theta}\}\leq\nn
	%\left\{\begin{array}{lcl}
	%2LR^2\Theta+16 MR\tau&\mbox{when}& M> \sigma\sqrt{N\over \tau},\\
	%40.5 \sigma R\sqrt{N}+16R\sigma\sqrt{\tau N}&\mbox{when}& M\leq \sigma\sqrt{N\over \tau}
	%\end{array}\right.\\
	&\leq&\max\{46 LR^2\tau,4LR^2\Theta,62\sigma R\sqrt{N\Theta},16\sigma R\sqrt{N\tau }\}
	\eqn{c'c''}
	for $1\leq \tau\leq N/\upsilon^2$. This implies \rf{eq:thm12}.\qed
	%\subsection{~Proof of Theorem~\ref{th:coset1sqrt}.}
	
	\textbf{A.5.~Proof of Theorem~\ref{th:coset1sqrt}.}
	We act in the same way as in the proof of Theorem~\ref{th:coset1} with the only difference that instead of part (i) of Lemma~\ref{le:mart1} we use part (ii) of that lemma, which implies that if $N\geq\upsilon^2$ then with probability at least $1-2e^{-\tau}$ the following inequalities hold:
	\[\begin{array}{rcl}
	\sum_{i=1}^N\la \xi_{i},z_{i-1}-x_{i-1}\ra&\leq& 8(1+\tau)R\max\{\sigma\sqrt{N},M\},\\
	\sum_{i=1}^N\|\xi_i\|_*^2&\leq& 8(2+3\tau)\max\{N\sigma^2,M^2\}.
	\end{array}
	\]
	
	%\subsection{~Proof of Proposition~\ref{pro:certif}.}
	\textbf{A.6.~Proof of Proposition~\ref{pro:certif}.}
	Define
	\bse
	\rho_N(\tau;\mu,\nu)&=&\nu^{-1}R^2\Theta+16 R\max\{\sigma\sqrt{N\tau},M\tau\}+\nn
	&&\quad\quad\quad\quad +\,
	20 (\mu+\nu)\max\{N\sigma^2,M^2\tau\}+
	\mu^{-1}\sum_{i=1}^NV_{x_{i-1}}(x_i).
	\ese
	The proposition is a direct consequence of the following result.
	\begin{lemma}\label{lem:cert1}
		Define
		\bqn
		\bar\rho_N(\tau)=\min_{\mu,\nu >0}\rho_N(\tau;\mu,\nu)&=&
		4R\sqrt{5\Theta\max\{N\sigma^2,M^2\tau\}}
		+16 R\max\{\sigma\sqrt{N\tau},M\tau\}+\nn
		&&+\,\min_{\mu > 0}\bigg\{ 20\mu\max\{N\sigma^2,M^2\tau\}+\mu^{-1}\sum_{i=1}^NV_{x_{i-1}}(x_i)\bigg\}.
		\eqn{eq:rhodef}
		Then, for $0<\tau\leq N/\upsilon^2$ and $t\ge L$ the following inequalities hold
		\bqn
		\begin{array}{lrcl}
			(a)\quad\quad\quad\quad&\Prob\left\{\epsilon_N(t)- \wh\epsilon_N(t)\geq \bar\rho_N(\tau)/N\right\}&\leq& 2\mathrm{e}^{-\tau},\\
			(b)\quad\quad\quad\quad&\Prob\left\{\wh\epsilon_N(t)-\epsilon_N(t)\geq \bar\rho_N(\tau)/N\right\}&\leq& 2\mathrm{e}^{-\tau}.\end{array}
		\eqn{eq:certstoch}
	\end{lemma}
	
	\begin{proofoflemma}
		{\!\!}~Let us prove the first inequality in \rf{eq:certstoch}. Recall that $\xi_i=y_i-\nabla\phi(x_{i-1})$, $i=1,...,N$.
		Due to the strong convexity of $ V_x(\cdot) $, for any $ z \in X $ and $ \mu > 0 $ we have
		\bse
		\la \xi_{i},z-x_{i}\rangle&=&\la \xi_{i},z-x_{i-1}\rangle+\la \xi_{i},x_{i-1}-x_i\rangle\leq\\
		&\leq& \la \xi_{i},z-x_{i-1}\rangle+\mbox{$\mu\over 2$}\|\xi_i\|_*^2+\mbox{$1\over \mu$}V_{x_{i-1}}(x_{i}).
		\ese
		Thus, for any $\nu>0$,
		\bse
		\sum_{i=1}^N\la \xi_{i},z-x_{i}\rangle&\leq &\sum_{i=1}^N\left[\la \xi_{i},z-x_{i-1}\rangle+\mbox{$\mu\over 2$}\|\xi_i\|_*^2+\mbox{$1\over \mu$}V_{x_{i-1}}(x_{i})\right]\leq\\
		&\leq &\mbox{$1\over \nu$}V_{x_0}(z)+\sum_{i=1}^N\left[\mbox{$\nu\over 2$}\|\xi_i\|_*^2+\la \xi_{i},z_{i-1}-x_{i-1}\rangle+\mbox{$1\over \mu$}V_{x_{i-1}}(x_{i})+\mbox{$\mu\over 2$}\|\xi_i\|_*^2\right]
		\ese
		(to obtain the last inequality, we have used Lemma~6.1 from \cite{nemirovski2009robust} with $ z_0 = x_0 $ in the same way as in the proof of the Proposition~\ref{pr:muprop}).
		By Lemma~\ref{le:mart1} there is a set $\cA_N$ of probability at least $ 1-2 \mathrm{e}^{-\tau} $ in the space of realizations $\omega^N$  such that, for all $\omega^N\in\cA_N$,
		\[
		\sum_{i=1}^N\la \xi_{i},z_{i-1}-x_{i-1}\ra\leq 16R\max\{\sigma\sqrt{N\tau},M\tau\}\;\;\mbox{and}\;\;
		\sum_{i=1}^N\|\xi_i\|_*^2\leq 40\max\{N\sigma^2,M^2\tau\}.
		\]
		Recalling that $V_{x_0}(z)\leq R^2\Theta$, we conclude that $  \sum_{i=1}^N\la \xi_{i},z-x_{i}\rangle
		\leq \rho_N(\tau;\mu,\nu)
		$ for all $z\in X$ and all $\omega^N\in \cA_N$.
		Therefore, for $\omega^N\in \cA_N$ we have
		\[
		\epsilon_N(t)-\wh\epsilon_N(t)=N^{-1}\sup_{z\in X}\sum_{i=1}^N\la \xi_{i},z-x_{i}\rangle\leq N^{-1}\min_{\mu,\nu\geq 0}\rho_N(\tau;\mu,\nu)=N^{-1}\bar\rho_N(\tau),
		\]
		which proves the first inequality in (\ref{eq:certstoch}). The proof of the second inequality in (\ref{eq:certstoch}) is similar and therefore it is  omitted.
		\qed
	\end{proofoflemma}

	\textbf{A.7.~Proof of Corollary~\ref{cor:smdb}.}
	From the definition of $\epsilon_N(\cdot)$ we deduce that
	\[
	\bar\beta\sum_{i=1}^N V_{x_{i-1}}(x_i)\leq \epsilon_N(\bar\beta),
	\]
	and we get \rf{barrho} by taking $\mu=1/\bar\beta$. On the other hand, one can check that for
	$\bar\beta\geq \max\left\{2L,{\sigma\sqrt{N}\over R\sqrt{\Theta}}\right\}$ the following inequalities hold:
	\bse
	\bar\rho(\tau)&\leq& N\epsilon_N(\bar\beta)+\max\Big\{
	[(20+4\sqrt{5})\sqrt{\Theta}+16\sqrt{\tau}] R\sigma\sqrt{N},\,
	(4\sqrt{5\Theta\tau}+26\tau)LR^2\Big\}\leq\\
	&\leq &N\epsilon_N(\bar\beta)+C_1R\sigma\sqrt{N[\Theta\vee \tau]}+C_2LR^2[\Theta\vee \tau].
	\ese
	Finally, since  $\wh\epsilon_N(\bar\beta)\leq\epsilon_N(\bar\beta)+\bar\rho(\tau)/N$  with probability at least $1-2\mathrm{e}^{-\tau}$ (cf. (\ref{eq:certstoch})(b)) we have
	\[\Delta_N(\tau,\bar\beta)=  \wh\epsilon_N(\bar\beta)+\bar\rho(\tau)/N\leq \epsilon_N(\bar\beta)+2\bar\rho(\tau)/N\]
	with the same probability. This implies \rf{eq:jnm}.\qed
	
	%%%%%%%%%%%%%%%%%%%%%%%%%%%%

	\textbf{A.8.~Proof of Theorem~\ref{th:strc}.}
	
	\par\noindent
	\textbf{1$^o$.}
	We first show that for each $k=1,\dots,m=m(N)$, the following is true.
	
	\noindent
	{\bf Fact ${I_k}$}.
	{\it There is a random event $\cB_k\subseteq \Omega^{\otimes N}$ of probability at least $1- 2k e^{-\tau}$ such that for all $\omega^N \in\cB_k$ the following inequalities hold:
		\bqn
		\begin{array}{lrcl}
			(a)\quad&\| y_{k} - \bar x(y_{k}) \|^2 &\leq& r^2_{k} = 2^{-k}r^2_0 \ \mbox{~~for some~} \ \bar x(y_{k})\in X_*,\\
			(b)\quad&F(y_k)-F_*&\le& {\kappa\over 2} r_k^2= 2^{-k-1}\kappa r_0^2.
		\end{array}
		\eqn{eq:rbe1}
	}
	The proof of Fact ${I_k}$ is carried out by induction. Note that (\ref{eq:rbe1})(a) holds with probability~1 for $k=0$. Set
	$\cB_0= \Omega^{\otimes N}$. Assume that (\ref{eq:rbe1})(a) holds for some $k\in \{0,\dots,m-1\}$ with probability at least $1-2ke^{-\tau}$, and let us show that then Fact ${I_{k+1}}$ is true.
	
	Define $F_*^k=\min_{x\in X_{r_{k}}(y_{k})} F(x)$ and let $X_*^k$ be the set of all minimizers of function $F$ on $X_{r_{k}}(y_{k})$.
	By Theorem~\ref{th:coset1} and the definition of $N_k$ (cf. \rf{N_k}), there is an event $\cA_k$ of probability at least $1-2e^{-\tau}$ such that for $\omega^N\in \cA_k$ after the $(k+1)$-th stage of the algorithm we have
	\bse
	{\kappa\over 2}\|y_{k+1}-\bar{x}_k(y_{k+1})\|^2\leq F(y_{k+1})-F_*^k&\leq&\max\left\{C_1{Lr^2_{k}[\tau\vee\Theta]\over N_{k+1}},C_2\sigma r_{k}\sqrt{[\tau\vee\Theta]\over N_{k+1}}\right\}\leq\\
	&\leq& {\kappa\over 4}r_{k}^2={\kappa\over 2}r_{k+1}^2,
	\ese
	where $\bar{x}_k(y_{k+1})$ is the projection of $y_{k+1}$ onto $X_*^k$.
	Set $\cB_{k+1}=\cB_{k}\cap \cA_k$. Then
	\[\Prob\{\cB_{k+1}\}\geq\Prob\{\cB_{k}\}+\Prob\{\cA_k\}-1\geq 1-{2(k+1) e^{-\tau}}.
	\]
	In addition, due to the assumption of induction, on the set $\cB_k$ (and, therefore, on $\cB_{k+1}$) we have
	\[
	\| y_{k} - \bar x(y_{k}) \| \leq r_{k},
	\]
	i.e., the distance between $ y_{k} $ and the set $X_*$ of global minimizers  does not exceed $ r_{k} $. Therefore, the set $X_{r_{k}}(y_{k})$ has a non-empty intersection with $X_*$. Thus, $X^k_*\subseteq X_*$, the point $\bar{x}_k(y_{k+1})$ is contained in $X_*$ and  $F_*^k$ coincides with the optimal value $F_*$ of the initial problem. We conclude that
	\[
	{\kappa\over 2}\|y_{k+1}-\bar x(y_{k+1})\|^2\leq F(y_{k+1})-F_*\le {\kappa\over 2} r_{k+1}^2= 2^{-k}\kappa r_0^2
	\]
	for some $\bar x(y_{k+1})\in X_*$.
	\par\noindent
	\textbf{2$^o$.} We now prove the theorem in the case  $\ovN_1\geq 1$. This condition is equivalent to the fact that $\ovN_k\geq 1$ for all $k=1,\dots,m(N)$, since  $\ovN_1\leq \ovN_2\leq \cdots\leq \ovN_{m(N)}$ by construction. Assume that $\omega^N\in \cB_{m(N)}$, so that \rf{eq:rbe1} holds with $k=m(N)$. Since $ \ovN_1 \geq 1$ we have $N_{k}\leq 2\ovN_{k}$. In addition, $\ovN_{k+1}\le 2\ovN_{k}$. Using these remarks and the definition of $m(N)$ we get
	\bqn
	N\leq \sum_{k=1}^{m(N)+1}N_k\le 2\sum_{k=1}^{m(N)+1}\ovN_k\le 2\sum_{k=1}^{m(N)}\ovN_k + 4 \ovN_{m(N)}
	\le 6\sum_{k=1}^{m(N)}\ovN_k.
	\eqn{eq:N}
	Thus, using the definition of $\ovN_k$ (cf. \rf{N_k}) we obtain
	\bse
	N&\leq& 6\sum_{k=1}^{m(N)}\max\left\{4C_1{L[\tau\vee\Theta]\over \kappa},\,16C_2{\sigma^2[\tau\vee\Theta]\over \kappa^2r^2_{k-1}}\right\}\leq\\
	&\leq&\underbrace{{24}\sum_{k=1}^{\bar k-1} {C_1L[\tau\vee\Theta]\over \kappa}}_{S_1}+\underbrace{{96}\sum_{k=\bar k}^{m(N)}{C_2\sigma^2[\tau\vee\Theta]\over \kappa^2r^2_{k-1}}}_{S_2},
	\ese
	where
	\[
	\bar{k}=\min\left\{k:\,{4C_2\sigma^2\over\kappa r_{k-1}^2}\geq C_1L
	\right\}.
	\]
	Two cases are possible: $S_1\geq N/2$ or $S_2\geq N/2$. If $S_1\geq N/2$, then
	\[
	\bar k\geq {C'\kappa N\over L[\tau\vee \Theta]},
	\]
	so that if $\omega^N\in \cB_{m(N)}$ then
	\bqn
	F(\wh{x}_N)-F_*\leq {\kappa\over 2}r^2_{m(N)}\leq {\kappa\over 2}r^2_{\bar k}=
	2^{-\bar k-1}{\kappa}r^2_{0}\leq C\kappa r_0^2\exp\left\{-{C'\kappa N\over L[\tau\vee \Theta]}\right\}.
	\eqn{1exp}
	If $S_2\geq N/2$ the following inequalities hold:
	\[
	{\kappa^2r^2_{0}N\over \sigma^2[\tau\vee\Theta]} \le  {C\kappa^2r^2_{0}\over \sigma^2[\tau\vee\Theta]}\sum_{k=\bar k}^{m(N)}2^{k}{\sigma^2[\tau\vee\Theta]\over \kappa^2r^2_{0}}\le C'2^{m(N)-\bar k} .
	\]
	Therefore, in this case for $\omega^N\in \cB_{m(N)}$ we have
	\bqn
	F(\wh{x}_N)-F_*\leq {\kappa\over 2}r^2_{m(N)}= {\kappa\over 2}r^2_{\bar k}2^{-m(N)+\bar k}\leq {\kappa\over 2}r^2_{0}2^{-m(N)+\bar k}\leq C{\sigma^2[\tau\vee \Theta]\over \kappa N}.
	\eqn{2exp}
	%%%
	%
	\par\noindent\textbf{3$^o$.}
	Finally, consider the case
	\bqn
	\ovN_1:=\max\left\{4C_1{L[\tau\vee\Theta]\over \kappa},\,16C_2{\sigma^2[\tau\vee\Theta]\over \kappa^2r^2_{0}}\right\}<1.
	\eqn{r0bound}
	Let $k_*\geq 2$ be the smallest integer $ k $ such that $\ovN_k\geq 1$.
	If $k_*> N/4$ it is not difficult to see that $m(N)\geq N/4$ and therefore for $\omega^N\in\cB_{m(N)}$ we have
	\bqn
	F(y_{m(N)})-F_*\leq {\kappa\over 2}r^2_{m(N)}
	\leq {\kappa\over 2}r^2_{0}\,2^{-N/4}.
	\eqn{exp3}
	If $2\le k_*\leq N/4$ we have the following chain of inequalities:
	$$
	3\sum_{k=k_*}^{m(N)} \ovN_k\ge \ovN_{m(N)+1} + \sum_{k=k_*}^{m(N)} \ovN_k = \sum_{k=1}^{m(N)+1} \ovN_k -\sum_{k=1}^{k_*-1} \ovN_k \ge \sum_{k=1}^{m(N)+1} \ovN_k - N/4\ge {N/4},
	$$
	where the first inequality uses the fact that $\ovN_{m(N)+1}\le 2\ovN_{m(N)}$ and the last inequality follows from the definition of $m(N)$. Based on this remark and on the fact that
	$\ovN_k/2^k \le \ovN_{k_*}/2^{k_*}$ for $k\ge k_*$ we obtain
	\bse
	{N\over {12}}&\leq&  \sum_{k=k_*}^{m(N)} \ovN_k\leq \sum_{k=k_*}^{m(N)}2^{k-k_*}\ovN_{k_*}\leq 2^{m(N)-k_*+2},
	\ese
	where the last inequality follows by noticing that $\ovN_{k_*}\le 2\ovN_{k_*-1} <2$. Hence, taking into account (\ref{eq:rbe1})(b) we get that, for $\omega^N\in\cB_{m(N)}$,
	\[
	F(\wh{x}_N)-F_*\leq \kappa r_{k_*}^22^{-m(N)+k_*}\leq \kappa r_{0}^22^{-m(N)+k_*}\leq C\kappa r_{0}^2/N.
	\]
	Combining this bound with \rf{1exp}, \rf{2exp} and \rf{exp3} we get
	\rf{th2bound}.
	\qed

\end{document}